\documentclass{amsart}
\usepackage{amscd,amssymb,amsopn,amsmath,amsthm,graphics,amsfonts,accents,enumerate,verbatim,calc}
\usepackage[dvips]{graphicx}
\usepackage{pgfplots}
\usepackage[pagebackref,colorlinks=true,linkcolor=red,citecolor=blue]{hyperref}
\usepackage[all]{xy}
\usepackage{cite}
\usepackage{tikz}
\usepackage{tikz-cd}
\usepackage{mathrsfs}

\calclayout

\newcommand{\rt}{\rightarrow}
\newcommand{\lrt}{\longrightarrow}

\newcommand{\st}{\stackrel}

\newcommand{\C}{\mathbb{C}}
\newcommand{\D}{\mathbb{D}}
\newcommand{\K}{\mathbb{K}}

\newcommand{\Z}{\mathbb{Z}}

\newcommand{\SC}{\mathscr{C}}

\newcommand{\SR}{\mathscr{R}}
\newcommand{\sS}{\mathscr{S}}
\newcommand{\ST}{\mathscr{T}}

\newcommand{\Mod}{{\mbox{-}\rm{Mod}}}

\newcommand{\prj}{{\mbox{-}\rm{prj}}}

\newcommand{\Flat}{\mbox{-}{\rm{Flat}}}
\newcommand{\Cot}{\mbox{-}{\rm{Cot}}}

\newcommand{\inc}{{\rm{inc}}}

\newcommand{\Add}{{\rm{Add}}}

\newcommand{\Mor}{{\rm{Mor}}}

\newcommand{\id}{{\rm{id}}}

\newcommand{\Coker}{{\rm{Coker}}}
\newcommand{\Ker}{{\rm{Ker}}}

\newcommand{\Prj}{\mbox{-}{\rm{Prj}}}

\newcommand{\SiWC}{{\sigma\text{-}\mathcal{WC}}}
\newcommand{\SIWC}{{\Sigma\text{-}\mathcal{WC}}}

\newcommand{\SiSF}{{\sigma\text{-}\mathcal{SF}}}

\newcommand{\SiOSF}{{\sigma\text{-}}{\mathcal{OSF}}}

\newcommand{\Tor}{{\rm{Tor}}}
\newcommand{\Hom}{{\rm{Hom}}}
\newcommand{\homf}{{\mathcal{H}}om}
\newcommand{\Ext}{{\rm{Ext}}}

\newcommand{\limt}{\underset{{\longrightarrow}}{\lim}}

\theoremstyle{plain}
\newtheorem{theorem}{Theorem}[section]

\newtheorem{lemma}[theorem]{Lemma}
\newtheorem{facts}[theorem]{Facts}
\newtheorem{proposition}[theorem]{Proposition}
\newtheorem{question}[theorem]{Question}

\newtheorem{notation}[theorem]{Notation}

\theoremstyle{definition}
\newtheorem{definition}[theorem]{Definition}

\newtheorem{remark}[theorem]{Remark}
\newtheorem{setup}[theorem]{Setup}
\newtheorem{s}[theorem]{}

\theoremstyle{plain}

\theoremstyle{definition}

\numberwithin{equation}{section}

\begin{document}

\title[Strongly flat modules via universal localization]{Strongly Flat Modules via Universal Localization}

\author[J. Asadollahi, R. Hafezi and S. Sadeghi]{Javad Asadollahi, Rasool Hafezi and Somayeh Sadeghi}

\address{Department of Pure Mathematics, Faculty of Mathematics and Statistics, University of Isfahan, P.O.Box: 81746-73441, Isfahan, Iran}
\email{asadollahi@sci.ui.ac.ir, asadollahi@ipm.ir }

\address{School of Mathematics and Statistics, Nanjing University of Information Science \& Technology, Nanjing, Jiangsu 210044, P.\,R. China}
\email{hafezi@nuist.edu.cn}

\address{School of Mathematics, Institute for Research in Fundamental Sciences (IPM), P.O.Box: 19395-5746, Tehran, Iran}
\email{somayeh.sadeghi@ipm.ir }

\makeatletter \@namedef{subjclassname@2020}{\textup{2020} Mathematics Subject Classification} \makeatother



\begin{abstract}
In this paper, we investigate a non-commutative version of strongly flat modules, which is based on the concept of universal localization introduced by Cohn. We consider a set $\sigma$ consisting of maps of finitely generated projective $R$-modules, where $R$ is not necessarily a commutative ring. Let $R_{\sigma}$ denote the universal localization of $R$ with respect to $\sigma$. The class of $\sigma$-strongly flat modules is defined as the left class in the cotorsion pair generated by $R_{\sigma}$. We examine the homotopy category of $\sigma$-strongly flat modules and demonstrate that the thick subcategory $\mathscr{S}_{\sigma}$, consisting of acyclic complexes, wherein all syzygies are $\sigma$-strongly flat, forms a precovering class within this homotopy category. This implies that the quotient map from $\mathbb{K}({\sigma\text{-}\mathcal{SF}})$ to $\mathbb{K}({\sigma\text{-}\mathcal{SF}})/\mathscr{S}_{\sigma}$ always has a fully faithful right adjoint.
\end{abstract}

\maketitle


\section{Introduction}
Let $R$ be a commutative ring and $S \subseteq R$ be a multiplicative subset. The class of $S$-strongly flat modules is defined as the left class in the cotorsion pair generated by $R_S$, the localization of $R$ with respect to $S$. This class of modules has been studied first by Trlifaj  \cite[Section 2]{Tr1}, for the case in which $R$ is a domain and $R_S$ is the quotient ring of $R$. Later, many authors in various papers considered arbitrary commutative rings and arbitrary multiplicative sets, even with zero-divisors.

The motivation for studying the class of $S$-strongly flat modules originated from approximation theory, specifically, the problem was to characterize the commutative rings for which this class is a covering class. Rings satisfying this property have been called $S$-almost perfect and have been characterized by many homological and ring-theoretic properties. See, for instance, \cite{BS, FS, P1, PS1, BP} and the survey article \cite{Sa}.

In the non-commutative case, this class is studied in \cite{FN1} and \cite{FN2}, when $R$ is a right Ore domain with classical right quotient ring $Q$. To this end, they considered the case of an extension of rings $R \subseteq Q$ for which the embedding $R \rt Q$ is a homological epimorphism, that is, $\varphi: R \rt S$ is an epimorphism in the category of rings and $\Tor^ R_i(Q,Q)=0$, for all $i \geq 1$, see \cite{GL}. The authors extended further results to the non-commutative case, when the modules under consideration are strongly flat, when the extension $Q$ of $R$ is the maximal flat epimorphic right ring of quotients of $R$, and when $R$ is a right noetherian right chain domain \cite{FN3}.

Our first aim in this paper is to treat the class of strongly flat modules in a more general localization theory. The full generality for non-commutative localization is that studied by Cohn \cite{Coh1} and \cite{Coh2}. Cohn localization is a significant construction in non-commutative ring theory that generalizes the concept of inverting elements to inverting matrices.

Let $R$ be a non-commutative ring (always associative with identity) and $\sigma$ be a set of maps between finitely generated projective (left) modules. Ring morphism $\varphi: R \rt R'$ is called $\sigma$-inverting if the induced morphism $R'\otimes_RP \rt R'\otimes_RQ$, is an isomorphism, for all morphisms $P \rt Q$ in $\sigma$. The collection of all $\sigma$-inverting morphisms is a category, and it was shown by Cohn \cite{Coh1} and Schofield \cite{Sch} that this category has an initial object. Its initial object, denoted by $R \rt R_{\sigma}$, or $R \rt \sigma^{-1}R$, is called the Cohn localization or the universal localization of $R$. See \cite[Theorem 2.46]{F} for a ring-theoretic description and \cite[\S 1]{Kr2} for a more categorical description of this notion.

It is known that, unlike the classical localization of commutative rings, in general $R_{\sigma}$ is not a flat $R$-module. However, if $\sigma$ admits the calculus of left fractions, then $R_{\sigma}$ is a flat $R$-module. In Section \ref{Sec 2} of the paper, we study the class of $\sigma$-strongly flat modules, investigate some of its properties, and provide a characterization in terms of the so-called $\sigma$-contramodule $R$-modules.

On the other hand, Neeman, in \cite{N2}, provided a new description for the homotopy category $\K(R\Prj)$, where $R\Prj$ denotes the category of projective $R$-modules. He demonstrated that this category is equivalent to the Verdier quotient of the homotopy category of flat $R$-modules, $\K(R\Flat)$, by a thick subcategory $\mathscr{S}$ of $\K(R\Flat)$. The subcategory $\mathscr{S}$ has been studied in deep by Neeman, see Facts 2.14 of \cite{N2} for six characterizations of it, for instance, it consists of all acyclic complexes with flat syzygies, the so-called pure acyclic complexes. Furthermore, in \cite{N3}, he showed that the quotient map from $\K(R\Flat)$ to $\K(R\Flat)/\mathscr{S}$ always has a right adjoint. This result consequently gives a new, fully faithful embedding of $\K(R\Prj)$ into $\K(R\Flat)$. To this end, he showed first that the inclusion of $\mathscr{S}$ into $\K(R\Flat)$, or more generally into $\K(R\Mod)$, has a right adjoint \cite[Theorem 3.1]{N3}.

According to \cite[Facts 2.8 (i)]{N2}, we know that the homotopy category of projective $R$-modules, denoted $ \K(R\Prj)$, is a well-generated triangulated category and thus satisfies Brown representability. Furthermore, the natural inclusion $ e: \K(R\Prj) \lrt \K(\SiSF) $ preserves coproducts. Hence, by \cite[Theorem 8.4.4]{N1}, we can conclude that the inclusion $e$ has a right adjoint $ e^*: \K(\SiSF) \lrt \K(R\Prj) $, which is also a Verdier  quotient.  Following Neeman \cite[Remark 2.12]{N3}, we denote the kernel of the functor $e^*$ by $\K(R\Prj)^{\perp}$, where orthogonal is taken in $\K(\SiSF)$. Hence, $\K(R\Prj)$  is equivalent to $\K(\SiSF)/\K(R\Prj)^{\perp}$.

Therefore, it is interesting to know the structure of objects in $\K(R\Prj)^{\perp}$. If $\K(R\Prj)$ is considered as a subcategory of $\K(R\Flat)$ and the orthogonal is taken in $\K(R\Flat)$, Neeman \cite{N2} studied this orthogonal and provides different characterizations of objects in $\K(R\Prj)^{\perp}$. In Section \ref{Sec 3}, we follow a similar argument to get characterizations for complexes in $\K(R\Prj)^{\perp}$, in case the orthogonal is taken in $\K(\SiSF)$. We show that under mild conditions, $\K(R\Prj)^{\perp}$ is equal to $\sS_\sigma$, the class of all $\sigma$-strongly flat complexes, i.e., the class of all acyclic complexes with all syzygies $\sigma$-strongly flat. Moreover, in Subsection \ref{Subsec 3.1}, it will be shown that the inclusion of $\sS_\sigma$ into either of the categories $\K(R\Flat)$ and $\K(R\Mod)$ always admits right adjoints.

Neeman, in \cite{N3}, showed that every object in $\K(R\Flat)$ admits an $\mathscr{S}$-precover. In order to constitute an $\mathscr{S}$-precover for an object $Z$ in $\K(R\Flat)$, he employed an auxiliary ring, as detailed in \cite[\S 2]{N3}. In the last section of the paper, we will utilize this auxiliary ring to construct objects in $\mathscr{S}_{\sigma}$.

Throughout, let $R$ be an associative ring with identity. Unless otherwise specified, by the $R$-module we mean a left $R$-module. The category of all left $R$-modules will be denoted by $R$-Mod.

\section{$\sigma$-strongly flat modules}\label{Sec 2}
In this section, our objective is to study the theory of strongly flat modules based on the universal localization introduced and studied by Cohn \cite{Coh1} and Schofield \cite{Sch}.

\subsection{Universal localization}
Let $\sigma$ be a set of maps of finitely generated projective $R$ modules, that is,
\[\sigma = \{s_i : P_i \rt Q_i \ | \ P_i \ {\rm and} \ Q_i \ {\rm are \ finitely \ generated \ projective} \ R\text{-}{\rm modules}\}.\]
The universal localization of $R$ with respect to $\sigma$ is a ring $R_\sigma$  with a ring homomorphism $\phi: R\rt R_\sigma$  such that
\begin{itemize}
\item[$(i)$] For every $s\in \sigma$, the morphism $R_\sigma\otimes_R s$ is an isomorphism, i.e. $\phi$ is $\sigma$-inverting.
\item[$(ii)$] For every ring homomorphism $\psi: R\lrt U$ such that $U\otimes_R s$ is an isomorphism, for all $s \in \sigma$, there exists a unique ring homomorphism $\overline{\psi}: R_\sigma\lrt U$ such that $\psi=\overline{\psi} \phi$.
\end{itemize}

For the construction of  $R_\sigma$, see \cite[Theorem 4.1]{Sch}. It is known that $\phi: R\lrt R_\sigma$  is an epimorphism of rings. The following equivalent conditions are satisfied, see \cite[Theorem 4.8]{Sch}.

\begin{itemize}
  \item [$(i)$] $\Ext^1_R = \Ext^1_{R_\sigma}$ on left $R_\sigma$ modules;
  \item [$(ii)$]  $\Tor^1_R(R_\sigma,R_\sigma) = 0$;
  \item [$(iii)$]  $\Ext^1_R = \Ext^1_{R_\sigma}$ on right $R_\sigma$ modules.
\end{itemize}
Moreover, \cite[Proposition 1.2, XI]{S} implies the following equivalent conditions.
\begin{itemize}
\item[$(i)$] $R_\sigma\otimes_R R_\sigma\lrt R_\sigma$ is an isomorphism;
\item[$(ii)$] $R_\sigma\otimes_R R_\sigma/\phi(R)=0$.
\end{itemize}

\begin{remark}\label{Classical Localization}
If $R$ is a commutative ring and $\sigma\subseteq R$ is a multiplicative set, then the classical localization of $R$ at $\sigma$ is an example of a universal localization. Indeed, we can view every element $x$ of $\sigma$ as a morphism $\lambda_x: R\lrt R$ in $R\prj$, where $R\prj$ denotes the category of finitely generated projective $R$-modules. Moreover, if $R$ is local, then the classes of universal localizations and classical localizations coincide.  However, for non-local commutative rings, the class of universal localizations is larger than that of classical localizations,  for instance, see \cite[Example 4.5]{MS}.
\end{remark}

\begin{remark}
Let $\SR$ be the full subcategory of $R\Mod$ consisting of $R$-modules $X$ such that $\Hom_R(s, X)$ is invertible for all $s\in\sigma$. Then there is an equivalence $\SR\cong R_\sigma\Mod$, \cite[\S 2.3]{Kr2}.
\end{remark}

\subsection{Calculus of fractions}
In general, the universal localization $R_\sigma$ is not a flat $R$-module.  However, if we assume that $\sigma$ admits a calculus of left (right) fractions, then $R_\sigma$ is a left (right) flat $R$-module, see Remark \ref{Flatness of Rsigma} below. In this subsection, we recall the notion of calculus of fractions, which helps to describe explicitly the localization of categories with respect to the specific class of morphisms. For more details see for instance \cite[\S 1.2]{Kr2}.

\begin{definition}\label{Def LF}
Let $\SC$ be a category and $\Mor(\SC)$ be the class of all morphisms in $\SC$. Let $S\subseteq\Mor(\SC)$. We say that the class $S$ admits a calculus of left fractions if the following conditions hold.
\begin{itemize}
  \item[$(i)$] For every object $X$ in $\SC$, the identity morphism $1_X$ is in $S$. The class $S$ is closed under compositions of morphisms.
 \item[$(ii)$] Let $X'\st{s}\leftarrow X\rt Y$ be a pair of morphisms such that $s\in S$. This can be completed to a commutative diagram
 \[\begin{tikzcd}
 &X\rar\dar{s} & Y\dar[dotted]{t}\\
 &X'\rar[dotted] &Y'
 \end{tikzcd}\]
such that $t\in S$.
 \item[$(iii)$] Let $f, g: X\rt Y$ be morphisms in $\SC$. Assume that there exists a morphism $s: X'\rt X$ in $S$ such that $f s=g s$. Then there exists a morphism $t: Y\rt Y'$ in $S$ such that $t f=t g$.
\end{itemize}
\end{definition}

Dually, we say that the class $S$ admits a calculus of right fractions if it admits a calculus of left fractions in the opposite category $\SC^{\rm{op}}$.

\begin{remark}\label{Flatness of Rsigma}
As it is mentioned in the beginning of this subsection, it is known that if $\sigma$ admits a calculus of left fractions, then $R_{\sigma}$ is a flat $R$-module, see e.g. Lemma 1.2.2 of \cite{Kr2}. In fact, in this case $R_{\sigma}$ is a colimit of the copies of $R$. Throughout, we use this fact without further reference.
\end{remark}

\subsection{$\sigma$-strongly flat modules}
Now we have the necessary background to define $\sigma$-strongly flat modules.

\begin{definition}
We say that an $R$-module $C$ is $\sigma$-weakly cotorsion if $\Ext^1_R(R_\sigma, C) = 0$. The class of $\sigma$-weakly cotorsion left $R$-modules will be denoted by $\SiWC$. An $R$-module $F$ is called $\sigma$-strongly flat if it is in ${}^{\perp}\SiWC$, where
\[{}^{\perp}\SiWC= \{M \in R\Mod  \ | \ \Ext^1_R(M,C)=0, \ {\rm for \ all} \ C \in \SiWC \}.\]
 The class of $\sigma$-strongly flat $R$-modules will be denoted by $\SiSF$. It follows from Theorem 6.11 of \cite{GT} that the pair $(\SiSF, \SiWC)$ is a complete cotorsion pair. Therefore, the class of $\SiSF$ is always precovering, while $\SiWC$ is preenveloping.
\end{definition}

\s Let $R\Flat$ denote the class of flat modules. The modules in the class
\[R\Cot=R\Flat^{\perp}=\{C \in R\Mod \ | \ \Ext^1(F,C)=0, \ {\rm for \ all} \ F \in R\Flat\}\]
are called cotorsion or Enochs cotorsion modules. It is known that $(R\Flat, R\Cot)$ is a complete cotorsion pair.

\begin{remark}\label{Flat}
Let $\sigma$ admit a calculus of left fractions. So $R_{\sigma}$ is a flat $R$-module. Hence, every (Enochs) cotorsion module is $\sigma$-weakly cotorsion. This, in turn, implies that every $\sigma$-strongly flat module is flat.
\end{remark}

\s \label{2.9}  It follows from \cite[Corollary 6.13]{GT} that the class $\SiSF$ consists of all summands of modules $G$ such that $G$ fits into an exact sequence of the form
\[0 \lrt F \lrt G \lrt L \lrt 0,\]
where $F$ is a free $R$-module and $L$ is $\{ R_\sigma \}$-filtered.

In fact, the paragraph before the statement of Lemma 6.15 of \cite{GT} and the fact that \[\Ext^1_R(P, Q)=0,\] for every projective $R_\sigma$-modules $P, Q$,  imply that a $\{R_\sigma\}$-filtered $R$-module is a free $R_\sigma$-module.
Thus every $\sigma$-strongly flat module $F$ is a direct summand of $R$-module $G$ such that $G$ fits into an exact sequence of the form
\[0 \lrt F \lrt G \lrt L \lrt 0\]
where $F$ is free $R$-module and $L$ is  a free $R_\sigma$-module, see also the paragraph above Lemma 3.4 of \cite{FN1}.

\begin{lemma}\label{Lemma 3.1[BP]}
Assume that $\sigma$ admits a calculus of left fractions. Let $M$ be a  $\sigma$-strongly flat $R$-module. Then
$R_\sigma\otimes_R M $ is a projective $R_\sigma$-module.
\end{lemma}

\begin{proof}
Since $M$ is a  $\sigma$-strongly flat $R$-module  then it is a direct summand of an $R$-module $G$, where $G$ fits into the short exact sequence
\[0\lrt R^{(\alpha)}\lrt G\lrt R_{\sigma}^{(\beta)}\lrt 0.\]
Now by tensoring the above sequence by $R_\sigma$, we get the result.
\end{proof}

\begin{setup}
Toward the end of this section, let $R$ be a commutative ring and let $\sigma$ be a set of maps between finitely generated projective $R$-modules.  Note that by Remark \ref{Classical Localization}, the class of universal localizations is larger than that of classical localizations.
\end{setup}

\s An $R$-module $M$ is called $\sigma$-contramodule if it is $\sigma$-$h$-reduced and $\sigma$-weakly cotorsion, that is if,
\[\Hom_R(R_\sigma, M) = 0 =\Ext^1_R(R_\sigma, M).\] An $R$-module $M$ is called $\sigma$-$h$-divisible if it is a quotient $R$-module of an $R_\sigma$-module, or equivalently, if every homomorphism $R\lrt M$ extends to a morphism $R_\sigma\lrt M$. Every $R$-module $M$ has a unique maximal $\sigma$-$h$-divisible submodule $h_\sigma(M)\subset M$ which can be constructed as the image of the natural map $\Hom_R(R_\sigma, M)\lrt M$, see \cite[\S 1]{P2}.

Following the argument used in \cite{BP}, we denote by $K^{\bullet}$ the two-term complex of $R$-modules \[R\st{\phi}\lrt R_\sigma\]
where the nonzero $R$-modules are in degrees $-1$ and $0$. Therefore, we see the complex $K^{\bullet}$ as an object of the derived category $\D(R\Mod)$.

We use the special notation \[\Ext^n_R(K^{\bullet}, M)=\Hom_{\D(R\Mod)}(K^{\bullet}, M[n]), ~~~~n\in\Z,\]
for any $R$-module $M$.

We have a distinguished triangle
\[R\st{\phi}\lrt R_{\sigma}\lrt K^{\bullet}\lrt R[1]\]
in the derived category $\D(R\Mod)$.

By applying the functor $\Hom_{\D(R\Mod)}(-, M[*])$ on the above triangle,  we have the following exact sequence,
\begin{equation}\label{Seq: I}
0\lrt \Ext^0_R(K^{\bullet}, M)\lrt \Hom_R(R_\sigma, M)\lrt M\lrt\Ext^1_R(K^{\bullet}, M)\lrt \Ext^1_R(R_\sigma, M)\lrt 0, \tag{I}
\end{equation}
which can be rewritten in the form of two short exact sequences
\begin{equation}\label{Seq: II}
0\lrt \Ext^0_R(K^{\bullet}, M)\lrt \Hom_R(R_\sigma, M)\lrt h_\sigma(M)\lrt 0,\tag{II}
\end{equation}
\begin{equation}\label{Seq: III}
0\lrt M/h_\sigma(M)\lrt \Ext^1_R(K^{\bullet}, M)\lrt \Ext^1_R(R_\sigma, M)\lrt 0.\tag{III}
\end{equation}

Using a similar argument as in the proof of \cite[Lemma 1.7 (c) and Lemma 2.4]{P2}, we can state the following lemma. Therefore, we will omit the proof.

\begin{lemma}\label{Lemma 1.7(c) and Lemma 2.4(1)}
The following holds true.
\begin{itemize}

\item[$(i)$] Assume that the projective dimension of $R_\sigma$ as an $R$-module is at most 1. Then, for every $R$-module $X$, the $R$-module $\Ext^1_R(K^{\bullet}, X)$ is  $\sigma$-contramodule.

\item[$(ii)$] Assume that  $M$ is  an $R$-module such that $R_\sigma\otimes_R M=0$.  Then, for every $R$-module $X$, the $R$-module $\Hom_R(M, X)$ is  $\sigma$-contramodule.
\end{itemize}
\end{lemma}

The following proposition provides a characterization of $\sigma$-strongly flat modules. Its proof is similar to the proof of Lemma 7.2 of \cite{BP}, which proves a similar result in the classical setting. For the reader's convenience, we provide a proof.

\begin{proposition}\label{Lemma 7.2[BP]}
Assume that the projective dimension of $R_\sigma$ as an $R$-module is at most 1. Let $M$ be an $R$-module such that $R_\sigma\otimes_R M$ is a projective $R_\sigma$-module. Then the $R$-module $M$ is $\sigma$-strongly flat if and only if $\Ext^1_R(M, X)=0=\Ext^2_R(M,X)$, for every $\sigma$-contramodule $R$-module $X$.
\end{proposition}

\begin{proof}
Since the projective dimension of $R_\sigma$ is at most 1, it follows from \ref{2.9} that the projective dimension of every $R_\sigma$-strongly flat $R$-module is also at most 1. Thus, the necessity condition follows.

To demonstrate the sufficient condition, let $C$ be a $\sigma$-weakly cotorsion $R$-module. We aim to show that $\Ext^1_R(M, C)=0$.
Consider short exact sequences
\begin{equation}\label{Seq: 1}
 0\lrt \Hom_R(R_\sigma/\phi(R), C)\lrt \Hom_R(R_\sigma, C)\lrt h_\sigma(C)\lrt 0\tag{1},
\end{equation}
and
\begin{equation}\label{Seq: 2}
0\lrt h_\sigma(C)\lrt C \lrt C/h_\sigma(C)\lrt 0\tag{2},
\end{equation}
where the short exact sequence \eqref{Seq: 1} is obtained by using the isomorphism $\Ext^0_R(K^\bullet, C)\cong \Hom_R(R_\sigma/\phi(R), C)$ of the short exact sequence (\ref{Seq: II}).
Since $C$ is $\sigma$-weakly cotorsion, we conclude from the sequence \eqref{Seq: I} that
\[C/h_\sigma(C)\cong \Ext^1_R(K^\bullet, C).\]
Since projective dimension of $R_\sigma$ is at most 1, Lemma \ref{Lemma 1.7(c) and Lemma 2.4(1)} $(i)$ implies that $C/h_\sigma(C)$ is $\sigma$-contramodule.
Moreover, since $R_\sigma\otimes_R R_\sigma/\phi(R)=0$,  Lemma \ref{Lemma 1.7(c) and Lemma 2.4(1)} $(ii)$ implies that the $R$-module $\Hom_R(R_\sigma/\phi(R), C)$ is a $\sigma$-contramodule.
Therefore, by assumption \[\Ext^1_R(M, C/h_\sigma(C))=0=\Ext^2_R(M, \Hom_R(R_\sigma/\phi(R), C)).\]

On the other hand, from the exact sequence \eqref{Seq: 1}, we get
\[\Ext^1_R(M, \Hom_R(R_\sigma, C))\lrt\Ext^1_R(M, h_\sigma(C))\lrt \Ext^2_R(M, \Hom_R(R_\sigma/\phi(R), C))=0.\]
But since by the assumption $R_\sigma\otimes_R M$ is projective, $\Ext^1_R(M, \Hom_R(R_\sigma, C))=0$. Thus, we obtain $\Ext^1_R(M, h_\sigma(C))=0$. Now from the exact sequence \eqref{Seq: 2}, we get the exact sequence
\[0=\Ext^1_R(M, h_\sigma(C))\lrt \Ext^1_R(M, C)\lrt \Ext^1_R(M, C/h_\sigma(C))=0.\]
Hence, the result follows, that is,  $\Ext^1_R(M, C)=0$.
\end{proof}

\section{Homotopy category of $\sigma$-strongly flat modules}\label{Sec 3}
According to \cite[Facts 2.8 (i)]{N2}, we know that the homotopy category of projective $R$-modules, denoted $ \K(R\Prj)$, is a well-generated triangulated category and thus satisfies Brown representability. Furthermore, the natural inclusion $ e: \K(R\Prj) \lrt \K(\SiSF) $ preserves coproducts. Hence, by \cite[Theorem 8.4.4]{N1}, we can conclude that the inclusion $e$ has a right adjoint $ e^*: \K(\SiSF) \lrt \K(R\Prj) $, which is also a Verdier  quotient.  Following Neeman \cite[Remark 2.12]{N3}, we denote the kernel of the functor $e^*$ by $\K(R\Prj)^{\perp}$, where orthogonal is taken in $\K(\SiSF)$. Hence, $\K(R\Prj)$  is equivalent to $\K(\SiSF)/\K(R\Prj)^{\perp}$.

Therefore, it is interesting to know the structure of objects in $\K(R\Prj)^{\perp}$. If $\K(R\Prj)$ is considered as a subcategory of $\K(R\Flat)$ and the orthogonal is taken in $\K(R\Flat)$, Neeman provided six characterizations of objects in $\K(R\Prj)^{\perp}$, see \cite[Facts 2.14]{N2}. Here we follow a similar argument to get characterizations for complexes in $\K(R\Prj)^{\perp}$, in case the orthogonal is taken in $\K(\SiSF)$.
We need the following remark.

\begin{remark}\label{Remark 2.15 Neeman}
It is proved by Neeman \cite[Remark 2.15]{N2} that every acyclic complex of projective modules with flat kernels is contractible.
\end{remark}

\setup
 In this section,  let $R$ be a commutative ring, let $\sigma$ be a set of maps between finitely generated projective $R$-modules and assume that $\sigma$ admits a calculus of left fractions.

\begin{theorem}\label{Characterization-I}
Assume that the projective dimension of $R_\sigma$ as an $R$-module is at most 1, and all flat $R$-modules have finite projective dimensions. Let $Z$ be an object of $\K(\SiSF)$. The following are equivalent.
\begin{itemize}
\item[$(i)$] $Z$ lies in the subcategory $\K(R\Prj)^{\perp}$ of $\K(\SiSF)$.
\item[$(ii)$] $Z$ is an acyclic complex of $\sigma$-strongly flat $R$-modules
\[\cdots\st{\partial^{i-1}}\lrt Z^{i-1}\st{\partial^{i}}\lrt Z^{i}\st{\partial^{i+1}}\lrt Z^{i+1}\lrt \cdots\]
such that the kernels $K^i$ of the maps $\partial^i$ are all $\sigma$-strongly flat $R$-modules.
\end{itemize}
\end{theorem}

\begin{proof}
$(i)\Rightarrow (ii)$. Since $\K(\SiSF) \subset \K(R\Flat)$, it follows that $Z \in \K(R\Flat)$. Therefore, it follows from \cite[Theorem 8.6]{N2}, that the complex $Z $ is an acyclic complex of $\sigma$-strongly flat $ R$-modules,
\[\cdots\st{\partial^{i-1}}\lrt Z^{i-1}\st{\partial^{i}}\lrt Z^{i}\st{\partial^{i+1}}\lrt Z^{i+1}\lrt \cdots\]
where the kernels $K^i$ of the maps $\partial^i$ are all flat $R$-modules. To complete the proof, we still need to demonstrate that all the kernels $ K^i $ are $\sigma$-strongly flat. To do this, we use Proposition \ref{Lemma 7.2[BP]}.
First observe that  since the module $Z^i$ for each $i$, is $\sigma$-strongly flat $R$-module, by Lemma \ref{Lemma 3.1[BP]}, $R_\sigma\otimes Z^i$ is a projective $R_\sigma$-module. Therefore, we obtain an acyclic complex
\[\cdots\st{}\lrt R_\sigma\otimes_R Z^{i-1}\st{}\lrt R_\sigma\otimes_R Z^{i}\st{}\lrt R_\sigma\otimes_R Z^{i+1}\lrt \cdots\]
of projective $R_\sigma$-modules such that all kernels are flat $R_\sigma$-modules. Consequently, as stated in  Remark \ref{Remark 2.15 Neeman}, this complex is contractible and hence, for each $i$, $R_\sigma\otimes_R \Ker\partial^i$ is a projective $R_\sigma$-module. Since the projective dimensions of flat $R$-modules are finite, we get that $\Ext^1_R(\Ker\partial^i, X)=0=\Ext^2_R(\Ker\partial^i, X)$ for every $\sigma$-contramodule $X$ and for  each $i$. Thus, for each $i$, $\Ker\partial^i$ is a $\sigma$-strongly flat $R$-module.

$(ii)\Rightarrow (i)$ Since $\SiSF\subset R\Flat$, the result follows by a similar argument as in the proof of $(iii)\Rightarrow (i)$ of \cite[Theorem 8.6]{N2}.
\end{proof}

The above theorem is proved under the additional assumption that the projective dimension of flat modules is finite. To eliminate this assumption, and inspired by \cite[1.1]{PS1}, we introduce a broader category than that of $\sigma$-strongly flat modules. We refer to this new class as optimistically $\sigma$-strongly flat $R$-modules, denoted by $\SiOSF$.

\begin{definition}
A flat $R$-module $F$ is called optimistically $\sigma$-strongly flat if $R_\sigma\otimes_R F$ is a projective $R_\sigma$-module and, for every $s\in \sigma$,  $\Coker s\otimes_R F\in \Add(\Coker s)$.
\end{definition}

\begin{remark}
Let $R$ be a commutative ring and $S$ be a multiplicatively closed subset of $R$. It is known that the classical localization of $R$ at $S$ is a universal localization, see Remark \ref{Classical Localization}. In this case, if we consider every element of $\sigma$ as a morphism $\lambda_x: R\lrt R$ in $R\prj$, then optimistically $\sigma$-strongly flat modules are exactly optimistically $S$-strongly flat modules, which were defined and studied in \cite{AS}.
\end{remark}

\begin{lemma}
Let $s: P\lrt Q$ be a morphism in $\sigma$. Then $R_\sigma\otimes_R\Coker s=0$.
\end{lemma}

\begin{proof}
Consider the exact sequence
\[ R_\sigma\otimes_R P \st{R_\sigma\otimes_Rs}\lrt R_\sigma\otimes_R Q \lrt R_\sigma\otimes_R \Coker s \lrt 0.\]
Since $R_\sigma\otimes_Rs$ is an isomorphism, we get the result.
\end{proof}

\begin{lemma}\label{SiSF IS OSiSF}
Every $\sigma$-strongly flat $R$-module $M$ is optimistically $\sigma$-strongly flat.
\end{lemma}

\begin{proof}
First, since $\sigma$ admits a calculus of left fractions, by Remark \ref{Flat}, $M$ is flat. By Lemma \ref{Lemma 3.1[BP]}, $R_\sigma\otimes_R M$ is a projective $R_\sigma$-module.  Consider the pure exact sequence
\[0\lrt R^{(\alpha)}\lrt G\lrt R_{\sigma}^{(\beta)}\lrt 0,\]
where $M$ is a summand of $G$. Let $s\in \sigma$. By tensoring  $\Coker s$ to the above short exact sequence, we obtain the exact sequence
\[0\lrt \Coker s\otimes_R R^{(\alpha)}\lrt \Coker s\otimes_R G\lrt \Coker s\otimes_R  R_{\sigma}^{(\beta)}\lrt 0.\]
Since $\Coker s\otimes_R  R_{\sigma}^{(\beta)}=0$, it follows that  $\Coker s\otimes_R G\cong (\Coker s) ^{(\alpha)}$. Hence, $\Coker s\otimes_R M\in \Add(\Coker s)$.
\end{proof}

\begin{remark}
Similar to \cite{N2}, we can see that the natural inclusion $ e: \K(R\Prj) \lrt \K(\SiOSF) $  has a right adjoint $ e^*: \K(\SiOSF) \lrt \K(R\Prj) $, which is also a Verdier  quotient. We denote the kernel of the functor $e^*$ by $\K(R\Prj)^{\perp}$, where orthogonal is taken in $\K(\SiOSF)$. Therefore,
$\K(R\Prj)$  is equivalent to $\K(\SiOSF)/\K(R\Prj)^{\perp}$.
\end{remark}

In the following theorem, we study the subcategory $\K(R\Prj)^{\perp}$ of $\K(\SiOSF)$.

\begin{theorem}\label{Characterization-II}
Let $\sigma$ be a set of maps between finitely generated projective $R$-modules such that for every $s\in \sigma$, $\Coker s$ is a projective $R$-module. Let $Z$ be an object of $\K(\SiOSF)$. The following are equivalent.
\begin{itemize}
\item[$(i)$] $Z$ lies in the subcategory  $\K(R\Prj)^{\perp}$ of $\K(\SiOSF)$.
\item[$(ii)$] $Z$ is an acyclic complex of optimistically $\sigma$-strongly flat $R$-modules
\[\cdots\st{\partial^{i-1}}\lrt Z^{i-1}\st{\partial^{i}}\lrt Z^{i}\st{\partial^{i+1}}\lrt Z^{i+1}\lrt \cdots\]
such that the kernels $K^i$ of the maps $\partial^i$ are all optimistically $\sigma$-strongly flat $R$-modules.
\end{itemize}
\end{theorem}

\begin{proof}
The proof is similar to the proof of Theorem \ref{Characterization-I}, so we just sketch the proof of $(i)\Rightarrow (ii)$. Since $\K(\SiOSF) \subset \K(R\Flat)$, by \cite[Theorem 8.6]{N2}, the complex $Z $ is an acyclic complex of optimistically  $\sigma$-strongly flat $ R$-modules,
\[\cdots\st{\partial^{i-1}}\lrt Z^{i-1}\st{\partial^{i}}\lrt Z^{i}\st{\partial^{i+1}}\lrt Z^{i+1}\lrt \cdots\]
where the kernels $K^i$ of the maps $\partial^i$ are all flat $R$-modules. Since the module $Z^i$ for each $i$, is optimistically $\sigma$-strongly flat, $R_\sigma\otimes Z^i$ is a projective $R_\sigma$-module. Therefore, by Remark \ref{Remark 2.15 Neeman}, for each $i$, $R_\sigma\otimes_R K^i$, is a projective $R_\sigma$-module.  By a similar argument and using the fact that $\Coker s$, for every $s\in \sigma$, is projective, we obtain that $\Coker s\otimes_R K^i$, for each $i$, is in $\Add(\Coker s)$. Hence, all kernels $K^i$ are optimistically $\sigma$-strongly flat.
\end{proof}

\subsection{On the existence of $\sS_\sigma$-precover}\label{Subsec 3.1}
Let $R$ be an arbitrary ring. Let $\sS_\sigma$ denote the thick triangulated subcategory of $\K(\SiSF)$ consisting of acyclic complexes $Z$ of $\sigma$-strongly flat $R$-modules
\[\begin{tikzcd}
 &\cdots\rar{\partial^{i-2}}& Z^{i-1}\rar{\partial^{i-1}}&Z^i\rar{\partial^i}& Z^{i+1} \rar{\partial^{i+1}}& \cdots
	\end{tikzcd}\]
in which all syzygies are $\sigma$-strongly flat $R$-modules.

In this subsection, we demonstrate that every object in the category $\K(R\Mod)$ has an $\sS_\sigma$-precover. To do this, we employ the method described in \cite{EEI}. First, we recall an important proposition. For reference, see \cite[Proposition 1.4]{N3} or \cite[Proposition 3]{Kr1}.

\begin{proposition}\label{Prop. 1.4 Neeman}
Let $\ST$ be a triangulated category and $\sS$ be a full triangulated subcategory. Assume that $\ST$ and $\sS$ have split idempotents. Then the following are equivalent.
\begin{itemize}
\item[$(i)$] The inclusion $\sS\lrt \ST$ has a right adjoint.
\item[$(ii)$] Every object $t\in \ST$ admits an $\sS$-precover.
\end{itemize}
\end{proposition}

\begin{lemma}\label{InclusionHasAdjoint}
The inclusion $\sS_\sigma\lrt \K(R\Mod)$ has a right adjoint. In particular, the inclusion $\sS_\sigma\lrt\K(\SiSF)$ has a right adjoint.
\end{lemma}

\begin{proof}
Since  $(\SiSF, \SiWC)$ is a complete cotorsion pair generated by $R_\sigma$ in $R\Mod$,  Theorem 3.1 of \cite{EEI} implies that $(\widetilde{\SiSF}, dg \widetilde{\SiWC})$ is a complete cotorsion pair in $\C(R\Mod)$, where $\widetilde{\SiSF}$ is the class of all acyclic complexes of $R$-modules with all syzygies in $\SiSF$ and $dg \widetilde{\SiWC}$ is the class of all complexes $X$ of $R$-modules such that all terms are in $\SiWC$ and the Hom complex $\homf(X, Y)$ is acyclic whenever $Y\in \widetilde{\SiSF}$. Hence, Theorem 3.5 of \cite{EBIJR} implies that the embedding \[\K(\widetilde{\SiSF})\lrt \K(R\Mod)\]  has a right adjoint.
Finally, note that $\K(\widetilde{\SiSF})=\sS_\sigma$.
\end{proof}

\begin{theorem}\label{SsigmaIsPrecovering}
Every object in the category $\K(R\Mod)$  has an $\sS_\sigma$-precover. In particular, every object in the category $\K(\SiSF)$ has an $\sS_\sigma$-precover.
\end{theorem}

\begin{proof}
The categories $\K(R\Mod)$ and $\sS_\sigma$ both  have coproducts. Therefore, \cite[Proposition 1.6.8]{N1} implies that idempotents split. Now, Lemma \ref{InclusionHasAdjoint} together with Proposition \ref{Prop. 1.4 Neeman} provides the result.
\end{proof}

\begin{remark}
$(i)$ With the assumptions in Theorem \ref{Characterization-I}, the projection $\K(\SiSF)$ to its Verdier quotient
$\K(\SiSF)/\K(R\Prj)^{\perp}$, has a right adjoint. Indeed, in this case, we have $\sS_\sigma=\K(R\Prj)^{\perp}$. Therefore, by Lemma \ref{InclusionHasAdjoint}, the  inclusion in the sequence
\[ \xymatrix{\sS_\sigma\ar@{^{(}->}[r] &\K(\SiSF)\ar@{->>}[r] &\K(\SiSF)/\sS_\sigma=\K(R\Prj)} \]
has a right adjoint. Consequently, the projection also possesses a right adjoint, as stated in \cite[Proposition 9.1.18]{N1}. Moreover,  since any adjoint, right or left, of a Verdier quotient is full and faithful, so we obtain a new fully faithful embedding of $\K(R\Prj)$ into $\K(\SiSF)$.

$(ii)$ With the assumptions in Theorem \ref{Characterization-II} and by the similar argument as stated in $(i)$, the natural projection $\K(\SiOSF)$ to its Verdier quotient $\K(\SiOSF)/\K(R\Prj)^{\perp}$, has a right adjoint. Therefore, there exists a new fully faithful embedding $\K(R\Prj)$ into $\K(\SiOSF)$.
\end{remark}

\section{Neeman's auxiliary algebra}\label{Sec 4}
Neeman, in \cite{N3}, showed that the quotient map from $\K(R\Flat)$ to $\K(R\Flat)/\mathscr{S}$ always has a right adjoint, where $\sS$ is the thick subcategory of $\K(R\Flat)$ consisting of all pure acyclic complexes. To this end, he first showed that the inclusion $\sS$ into $\K(R\Flat)$ has a right adjoint. To show the existence of this adjoint, in view of Proposition \ref{Prop. 1.4 Neeman}, he showed that every object in $\K(R\Flat)$ admits an $\sS$-precover. And to constitute an $\sS$-precover for an object $Z$ in $\K(R\Flat)$, he considered an auxiliary ring, see \cite[\S 2]{N3}.

In this section, we use this auxiliary ring to construct objects in $\sS_{\sigma}$, where $\sS_\sigma$ is the class of all acyclic complexes with all syzygies $\sigma$-strongly flat, which is defined in Subsection \ref{Subsec 3.1}.

Let us first review Neeman's construction as presented in \cite{N3}. Let $R$ be a ring and
\[\begin{tikzcd}
 & \cdots \ar{r}{\partial^{i-2}} &\cdot^{i-1}
\ar{r}{\partial^{i-1}}& \cdot^i\ar{r}{\partial^{i}}& \cdot^{i+1}\ar{r}{\partial^{i+1}}& \cdot^{i+2}\ar{r}{\partial^{i+2}}& \cdots
	\end{tikzcd}\]
be a quiver with relations $\partial^{i+1}\partial^i=0$, for all $i \in \Z$. Let $T=T(R)$ be the free $R$-module with basis $\lbrace 1, \partial^i, e^j\rbrace$, with $i, j\in\mathbb{Z}$. For further details on the $R$-algebra structure of $T$, please refer to \cite[\S 2]{N3}.

Let $\inc$ be the inclusion functor
\[\inc: \C(R\Mod )\lrt T\Mod \]
such that takes a complex $Z$ to $\bigoplus_{i=-\infty}^{\infty} Z^i$ with the natural $T$-module structure. This functor is fully faithful and has a right adjoint, which is denoted  by
\[C: T\Mod \lrt \C(R\Mod).\]
The functor $C$ takes the $T$-module $M$ to the complex
\[\begin{tikzcd}
 &\cdots\rar{\partial^{i-2}}& e^{i-1}M\rar{\partial^{i-1}}&e^{i}M\rar{\partial^i}& e^{i+1}M \rar{\partial^{i+1}}& \cdots.
	\end{tikzcd}\]

In the following, we recall some properties of the functors $\inc$ and $C$ from \cite[\S 2]{N3}.

\begin{facts}\label{Facts}
The following holds true.
\begin{itemize}
\item[$(i)$] If $Z$ is a contractible complex of projective $R$-modules, then $\inc(Z)$ is a projective $T$-module. In particular, if $Z$ is the complex
\[\overline{R}[i]:~~~\cdots\lrt 0\lrt R\st{1}\lrt R \lrt0\lrt \cdots\]
with the nonzero modules in degrees $i$ and $i+1$, then the projective module $T e^i$ agrees with $\inc(Z)$.
\item[$(ii)$] If $Z$ is an acyclic complex of flat $R$-modules with all syzygies flat, then $\inc(Z)$ is a flat $T$-module.
\item[$(iii)$]The functor $C$ is exact and preserves colimits.
\item[$(iv)$] If $P$ is a projective $T$-module, then the complex $C(P)$ is a contractible complex of projective $R$-modules. In particular, $C(T)$ is the coproduct of all suspensions of the complex
\[\overline{R}:~~~\cdots\lrt 0\lrt R\st{1}\lrt R \lrt0\lrt \cdots.\]
\item[$(v)$] If $F$ is a flat $T$-module, then  the complex $C(F)$ is an acyclic complex of  flat $R$-modules with all kernels flat $R$-modules.
\end{itemize}
\end{facts}

\begin{notation}
Let $\sigma$ be a set of maps between  finitely generated projective $R$-modules, that is, $\sigma\subseteq \Mor(R\prj )$.  We may assume that $\sigma$ contains the identity morphism of each object and if $s, t\in \sigma$ then $s\oplus t\in \sigma$, \cite[\S 2.3]{Kr2}. Let $s: P\lrt Q$ be a morphism in $\sigma$. By $\overline{s}$ we mean the morphism between complexes $\overline{P}$ and $\overline{Q}$, that is,
\[\begin{tikzcd}
\overline{P}\dar{\overline{s}}:  &\cdots\rar& 0\rar&P\rar{1}\dar{s}& P \rar\dar{s}& 0\rar& \cdots\\
\overline{Q}: &\cdots\rar& 0\rar&Q\rar{1}& Q \rar& 0\rar& \cdots
	\end{tikzcd}\]
where $\overline{P}$ and $\overline{Q}$ are complexes with nonzero modules in degrees $0$ and $1$. We note that $\overline{s}\subseteq \Mor(\C(R\Mod )\prj)$, where $\C(R\Mod )\prj$ denotes the class of all finitely generated projective objects in $\C(R\Mod)$. Let  $\overline{\sigma}$ denote the set of all  suspensions $\overline{s}[j]$  of the morphisms $\overline{s}$, that is,
\[\overline{\sigma}=\{ \overline{s}[j]~~~| ~~~s\in \sigma, ~j\in \Z\}.\]
Set
\[ \Sigma:=\inc(\overline{\sigma}).\]
Note that Facts \ref{Facts} $(i)$ implies that $\Sigma=\inc(\overline{\sigma})\subseteq \Mor(T\prj )$, where $T\prj$ denote the class of finitely generated $T$-modules. Moreover, we may assume that $\Sigma$ contains the identity morphism of each object in $T\prj$.
\end{notation}

Throughout the section, the sets $\sigma$ and $\Sigma$ are those introduced in the notation above.

\begin{proposition}\label{LF2&LF3 for T}
Let $\sigma\subseteq \Mor(R\prj)$ admit a calculus of left fractions.
\begin{itemize}
\item[$(i)$] Let   $X'\st{\tau}\leftarrow X\st{f}\rt T$ be a pair of morphisms in $T\prj $ such that $\tau\in \Sigma$. Then it can be completed to a commutative diagram
\[\begin{tikzcd}
 &X\rar{f}\dar{\tau} & T\dar[dotted]{\lambda}\\
 &X'\rar[dotted] &Y'
 \end{tikzcd}\]
 such that $\lambda\in \Sigma$.

\item[$(ii)$] Let $f, g: X\lrt T$ be morphisms in $T\prj $. If  there exists morphism $\tau: X'\rt X$ in $\Sigma$ such that $f \tau=g \tau$, then there exists a morphism $\lambda: T\rt Y'$ in $\Sigma$ such that $\lambda f=\lambda g$.
\end{itemize}
\end{proposition}

\begin{proof}
$(i)$ Since $\tau\in \Sigma$, without loss of generality, we may assume that $X=\inc(\overline{P})$ and $X'=\inc(\overline{Q})$, where $P, Q\in R\prj $. By applying the functor $C$ on the pair of morphisms $X'\st{\tau}\leftarrow X\st{f}\rt T$, and using the fact that $C\circ \inc = \id$, we obtain a pair of morphisms $\overline{Q}\st{\overline{\varepsilon}}\leftarrow \overline{P}\st{C({f})}\rt C(T)$ in $\C(R\Mod)\prj$, where $\tau=\inc(\overline{\varepsilon})$ and $\varepsilon\in\sigma$.
By using Facts \ref{Facts} $(iv)$ and noting that the complex $\overline{P}$ is nonzero in two degrees, we can easily verify that $C(T)$ agrees with $\overline{R}$.

By visualizing the above pair of morphisms in the nonzero degrees and using the fact that $\sigma$ admits a calculus of left fractions, we obtain the diagram
\[\xymatrix@!0{
& P \ar@{->}[rr]^{}\ar@{->}'[d][dd]& & R \ar@{.>}[dd]^{\zeta}\\
P\ar@{->}[ur]^{1} \ar@{->}[rr]\ar@{->}[dd]^{\varepsilon}& & R \ar@{->}[ur]^{1}\ar@{.>}[dd]\\
& Q \ar@{.>}'[r][rr]& & Y\\
Q\ar@{.>}[rr]^{}\ar@{->}[ur]& & Y \ar@{.>}[ur]^{1}
}\]
such that $\zeta\in\sigma$. Hence, we get the commutative diagram

\[\xymatrix{
 \overline{P} \ar@{->}[r]\ar@{->}[d]^{\overline{\varepsilon}} & \overline{R} \ar@{->}[d]^{\overline{\zeta}}\\
\overline{Q} \ar@{->}[r]&  \overline{Y} \\
}\]
in $\C(R\Mod)$.

 Now, by applying the functor $\inc$ on the above diagram, we get the commutative diagram
 \[\xymatrix{
 X=\inc(\overline{P}) \ar@{->}[r]\ar@{->}[d]^{\tau} &  \inc(\overline{R})=Te^0 \ar@{->}[d]^{\inc(\overline{\zeta})}\\
X'=\inc(\overline{Q}) \ar@{->}[r]&  \inc(\overline{Y})
}\]
in $T\Mod$ such that $\inc(\overline{\zeta})\in\Sigma$. To complete the proof, it suffices to demonstrate that the morphism $X\lrt Te^0$ is equal to the initial morphism $f: X\lrt T$. This follows from the following sequence of isomorphisms
\begin{align*}
\Hom_T(\inc(\overline{P}), T)&\cong \Hom_{\C(R\Mod )}(\overline{P}, C(T))\\
&=\Hom_{\C(R\Mod )}(\overline{P}, \overline{R})\\
& \cong \Hom_T(\inc(\overline{P}), \inc(\overline{R}))\\
&=\Hom_T(X, Te^0).
\end{align*}

The first isomorphism arises from the adjoint pair $(\inc, C)$. As in the previous part, the reader can easily verify that the second equality holds. The third equality follows from the fully faithful nature of the functor $\inc$, and the final equality is established by Facts \ref{Facts} (i).

$(ii)$ Assume there exists a morphism $\tau: X'\rt X$ in $\Sigma$ such that $f \tau=g\tau$, where $f$ and $g$  are depicted in the diagram
\[\xymatrix{X'\ar[r]^{\tau}& X\ar @{->} @< 2pt> [r]^{f} \ar @{ ->} @<-2pt> [r]_{g}&T.
}\]
 Since $\tau\in \Sigma$, without loss of generality, we assume that $X'=\inc(\overline{Q})$ and $X=\inc(\overline{P})$, where $Q, P\in R\prj $.
By applying the functor $C$, we get
\[\xymatrix{\overline{Q}\ar[r]^{\overline{\varepsilon}}& \overline{P}\ar @{->} @< 2pt> [r]^{C(f)} \ar @{ ->} @<-2pt> [r]_{C(g)}&C(T)
}\]
in $\C(R\Mod)$, where $\tau=\inc(\overline{\varepsilon})$ and $\varepsilon\in\sigma$.
Similar to the part $(i)$, we may consider $C(T)$ as  $\overline{R}$.  By considering the above diagram of complexes in nonzero degrees and using the fact that $\sigma$  admits a calculus of left fractions, we obtain the following diagram
\[\xymatrix{Q\ar[r]^{\varepsilon}\ar[d]^{1}& {P}\ar@{->} @< 2pt> [r]^{f'} \ar @{ ->} @<-2pt> [r]_{g'}\ar[d]^{1}&R\ar[d]^{1}\ar@{.>}[r]^{\zeta}&Y\ar@{.>}[d]^{1}\\
Q\ar[r]^{\varepsilon}& {P}\ar @{->} @< 2pt> [r]^{f'} \ar @{ ->} @<-2pt> [r]_{g'}&R\ar@{.>}[r]^{\zeta}&Y
}\]
where $\zeta\in \sigma$ and $\zeta f'=\zeta g'$.  By applying  the functor $\inc$ on the above diagram, we get
\[\xymatrix{X'\ar[r]^{\tau}& X \ar @{->} @< 2pt> [r]^{f_1} \ar @{ ->} @<-2pt> [r]_{g_1} & Te^0 \ar@{.>}[r]^{\lambda}&\inc(\overline{Y})}\]
such that $\lambda=\inc(\overline{\zeta})\in \Sigma$ and $\lambda\inc(\overline{f'})=\lambda\inc(\overline{g'})$.
Now the same series of isomorphisms used in the proof of Statement $(i)$, applies to show that $f_1$ and $g_1$ can be replaced by $f$ and $g$, respectively. The proof is hence complete.
\end{proof}

\begin{remark}\label{LF2&LF3 for T^n}
In the statements of the above proposition, if we replace the module $T$ with $T^{(n)}$, we can still obtain the result. Indeed, the category $\Sigma$ is closed under finite coproducts.
\end{remark}

\begin{proposition}\label{LF for Sigma}
If $\sigma\subseteq \Mor(R\prj)$ admits a calculus of left fractions, then so does $\Sigma\subseteq \Mor(T\prj )$.
\end{proposition}

\begin{proof}
According to the definition of  $\Sigma$, it is clear that the identity morphism of each object in $T\prj$ is included in $\Sigma$ and that $\Sigma$ is closed under compositions. Therefore, condition $(i)$ of Definition \ref{Def LF} is satisfied. Since every finitely generated projective $T$-module is a direct summand of $T^{(n)}$, it remains to show that conditions $(ii)$ and $(iii)$ of Definition \ref{Def LF} are fulfilled for $Y = T^{(n)}$. Proposition \ref{LF2&LF3 for T} and Remark \ref{LF2&LF3 for T^n} support this conclusion.
\end{proof}

\begin{proposition}\label{SigmaSF under functor C}
Let $\sigma$ admit a calculus of left fractions. Let $F$ be a $\Sigma$-strongly flat $T$-module. Then the complex $C(F)$ is in $\sS_\sigma$.
\end{proposition}

\begin{proof}
First, we note that since $\sigma$ admits a calculus of left fractions, so does $\Sigma$. Consequently, $F$ is a flat $T$-module. Hence, by Facts \ref{Facts} $(v)$, $C(F)$ is an acyclic complex of flat modules with flat syzygies. Since $F$ is $\Sigma$-strongly flat, it is a direct summand of a $T$-module $G$, where $G$ fits into the following short exact sequence
\[0\lrt T^{(n)}\lrt G\lrt T_{\Sigma}^{(m)}\lrt 0.\]
Given that $\Sigma$ admits a calculus of left fractions, we can deduce from \cite[Lemma 1.2.2]{Kr2} that $T_\Sigma$ is isomorphic to a colimit of copies of $T$; that is, we have $T_\Sigma \cong \limt T$.
Since $C$ is an exact functor, it preserves coproducts and colimits. Therefore, we obtain the short exact sequence
\[ 0\lrt C(T)^{(n)}\lrt C(G)\lrt \limt C(T)^{(m)}\lrt 0.\]
It is known that $C(T)$ is the coproduct of all the suspensions of the complex
\[\cdots \lrt 0\lrt R\lrt R\lrt 0\lrt \cdots,\]
and $\limt C(T)$ is the coproduct of all suspensions of the complex
\[\cdots \lrt 0\lrt R_\sigma\lrt R_\sigma \lrt 0\lrt \cdots.\]
This implies that for every $i \in \mathbb{Z}$, there exists an exact sequence
\[0\lrt R^{(n)}\lrt e^i G\lrt R_\sigma^{(m)}\lrt 0. \]
Therefore, for each $i \in \Z$, $e^iG$ is a $\sigma$-strongly flat $R$-module.
It is straightforward to see that all kernels of the maps $e^iG\lrt e^{i+1}G$, are also $\sigma$-strongly flat. Consequently, $C(F)$ is an acyclic complex of $\sigma$-strongly flat $R$-modules with all syzygies also being $\sigma$-strongly flat.
\end{proof}

We end this section by studying inc a little.

\begin{proposition}
Let $\overline{R_\sigma}$ be the complex \[\cdots \lrt 0\lrt R_\sigma\lrt R_\sigma \lrt 0\lrt \cdots\] in $\C(R\Mod )$. Then $\inc(\overline {R_\sigma})$ is a $\Sigma$-strongly flat $T$-module.
\end{proposition}

\begin{proof}
Similar to the paragraph above Corollary 7.42 of  \cite{GT},  we can see that $(T_\Sigma\Mod )^{\perp}=\SIWC$. Hence, to demonstrate that $\inc(\overline{R_\sigma})$ is a $\Sigma$-strongly flat $T$-module, it suffices to show that $\inc(\overline{R_\sigma})$ is a $T_\Sigma$-module. To establish this, we need to verify that
$\Hom_T(\inc(\overline{s}[j]), \inc(\overline{R_\sigma}))$ is invertible for all $s\in \sigma$. However, for every $s\in \sigma$, the fully faithful nature of the functors  $\inc$ and  $\overline{\iota}: R\lrt \C(R\Mod)$ imply the following isomorphisms
\[\Hom_T(\inc(\overline{s}[j]), \inc(\overline{R_\sigma}))\cong \Hom_{\C(R\Mod)}(\overline{s}[j], \overline{R_\sigma})\cong \Hom_R(s, R_\sigma).\]
Since $\Hom_R(s, R_\sigma)$ is invertible for every $s\in \sigma$, it follows that $\inc(\overline{R_\sigma})$ is indeed a $T_\Sigma$-module.
\end{proof}

\begin{question}
Is it true that for every $Z \in \mathcal{S}_\sigma$, the inclusion $\inc(Z)$ is a $\Sigma$-strongly flat $T$-module?
\end{question}

\section*{Acknowledgments}
We would like to thank Amnon Neeman for suggesting the study of strongly flat modules through Cohn localization and for introducing us to the concept of Cohn localization. Part of this work was conducted during the first author’s visit to the Institut
des Hautes \'{E}tudes Scientifiques (IHES) in Paris, France. He expresses his sincere gratitude for the support
and the enriching academic environment provided by IHES. The research of the third author was supported by a grant from IPM.

\end{document}